\documentclass{amsart}
\usepackage{verbatim}
\usepackage{rotating} 
\usepackage{amssymb}
\usepackage{mathrsfs,amsfonts,amsmath,amssymb,epsfig,amscd,xy,amsthm,pb-diagram} 

\newtheorem{theorem}{Theorem}[section]
\newtheorem{proposition}[theorem]{Proposition}
\newtheorem{notation}[theorem]{Notation}
\newtheorem{lemma}[theorem]{Lemma}

\newtheorem{corollary}[theorem]{Corollary}

\newtheorem{definition}[theorem]{Definition}

\theoremstyle{plain}

\theoremstyle{remark}

\newtheorem{remark}[theorem]{Remark}
\newtheorem{example}[theorem]{Example}

\newcommand{\C}{{\mathbb C}}

\newcommand{\Q}{{\mathbb Q}}

\newcommand{\R}{{\mathbb R}}

\newcommand{\T}{{\mathbb T}}
\newcommand{\Z}{{\mathbb Z}}
\newcommand{\N}{{\mathbb N}}

\newcommand{\cA}{{\mathcal A}}
\newcommand{\cE}{{\mathcal E}}

\newcommand{\fp}{\mathfrak p}

\DeclareMathOperator{\GF}{GF}

\newcommand{\Qbar}{\bar{\Q}}

\DeclareMathOperator{\lcm}{lcm}

\DeclareMathOperator{\Norm}{N}

\DeclareMathOperator{\ord}{ord}

\DeclareMathOperator{\card}{card}

\newcommand{\bF}{{\mathbb F}}

\newcommand{\cO}{\mathcal{O}}

\newcommand{\cP}{\mathcal{P}}

\DeclareMathOperator{\den}{den}

\author{Jason P.~Bell}
\address{
Jason P.~Bell\\
University of Waterloo\\
Department of Pure Mathematics\\
Waterloo, Ontario, Canada N2L 3G1}
\email{jpbell@uwaterloo.ca}

\author{Keira Gunn}
\address{
Keira Gunn \\
Department of Mathematics and Statistics\\
University of Calgary\\
AB T2N 1N4, Canada
}
\email{keira.gunn1@ucalgary.ca}

\author{Khoa D.~Nguyen}
\address{
Khoa D.~Nguyen \\
Department of Mathematics and Statistics\\
University of Calgary\\
AB T2N 1N4, Canada
}
\email{dangkhoa.nguyen@ucalgary.ca}

\author{J.~C.~Saunders}
\address{
J.~C.~Saunders \\
Department of Mathematics and Statistics\\
University of Calgary\\
AB T2N 1N4, Canada
}
\email{john.saunders1@ualgary.ca}


\begin{document}
	\title[Criterion for the P\'olya-Carlson dichotomy and application]{A general criterion for the P\'olya-Carlson dichotomy and application}
	
	\date{May 2022}
	
	\begin{abstract}
	We prove a general criterion for an irrational power series
	$f(z)=\displaystyle\sum_{n=0}^{\infty}a_nz^n$ with coefficients in a number field $K$
	to admit the unit circle as a natural boundary. 
	As an application, let $F$ be a finite field, let $d$ be a positive integer, let
	$A\in M_d(F[t])$ be a $d\times d$-matrix with entries in $F[t]$, and let 
	$\zeta_A(z)$ be the Artin-Mazur zeta function associated to the 
	multiplication-by-$A$ map on the compact abelian group
	$F((1/t))^d/F[t]^d$. We provide a complete characterization of when $\zeta_A(z)$ 
	is algebraic and prove that it admits the circle of convergence 
	as a natural boundary in the transcendence case. 
	This is in stark contrast to the case of linear endomorphisms on 
	$\mathbb{R}^d/\mathbb{Z}^d$ in which Baake, Lau, and Paskunas prove that 
	the zeta function is always rational. Some connections to earlier work of 
	Bell, Byszewski, Cornelissen, Miles, Royals, and Ward are discussed. Our method uses a similar technique in recent work of Bell, Nguyen, and Zannier together with certain patching arguments involving linear recurrence sequences.   
	\end{abstract}
	
	\maketitle
	
	\section{A general criterion for the P\'olya-Carlson dichotomy}
	Throughout this paper, let $\N$ denote the set of positive integers and let $\N_0=\N\cup\{0\}$. 
	We begin with the well-known P\'olya-Carlson dichotomy \cite{Car21_U}:
	\begin{theorem}\label{thm:PC}
	A power series $f(z)=\displaystyle\sum a_nz^n\in\Z[[x]]$ that converges inside the unit disk is either rational or it admits the unit circle as a natural boundary. Moreover, if $f(z)$ is rational then each pole is located at a root of unity.
	\end{theorem}
	
	For an algebraic number $\alpha$, we define its denominator, denoted $\den(\alpha)$, to be the smallest positive integer $d$ such that $d\alpha$ is an algebraic integer. Our first main result is the following:
	\begin{theorem}\label{thm:main 1}
	Let $S$ be a subset of $\N$ such that $\vert S\cap [1,n]\vert=o(n/\log n)$ as $n\to\infty$. Let $K$ be a number field and let $f(z)=\displaystyle\sum a_nz^n\in K[[x]]$
	such that $\sigma(f):=\displaystyle\sum \sigma(a_n)z^n$ converges in the open unit disk for every
	embedding $\sigma:\ K\rightarrow\C$. Suppose that for every $\beta>1$, we have
	\begin{equation}\label{eq:thm main 1}
	\lcm\{\den(a_k):\ k\leq n, k\notin S\}<\beta^n
	\end{equation}
	for every sufficiently large integer $n$. Then either $f(z)$ admits the unit circle as a natural boundary or there exists $\displaystyle\sum b_nz^n\in K[[z]]$ that is the power series of a rational function whose poles are located at the roots of unity such that $a_n=b_n$ for every $n\in \N\setminus S$.
	\end{theorem}
	
	\begin{remark}\label{rem:1 after main 1}
	Theorem~\ref{thm:PC} is a special case of 
	Theorem~\ref{thm:main 1} when $K=\Q$ and  $S=\emptyset$. Any $\sum b_nz^n$ as in the conclusion of Theorem~\ref{thm:main 1} satisfies that $\den(b_n)$ is bounded, hence the LHS of \eqref{eq:thm main 1} is bounded. Without any further assumption on the $a_n$'s for $n\in S$, the conclusion $a_n=b_n$ for $n\in \N\setminus S$ is best possible given examples such as 
	$\displaystyle\frac{1}{1-z}+\sum_{n\in S}\frac{z^n}{n!}$ that is convergent in the open unit disk and extends to an analytic function on the whole complex plane except at $z=1$. 
	\end{remark}
	
	\begin{remark}\label{rem:2 after main 1}
	Although we do not know whether the upper bound $o(n/\log n)$ for $\vert S\cap [1,n]\vert$ can be improved, it is at least very close to being optimal: we explain why Theorem~\ref{thm:main 1} is no longer valid when $o(n/\log n)$ is replaced by $o(n/(\log n)^{1-\epsilon})$ for any $\epsilon>0$.
	 Consider
	$$\sum a_nz^n=\log(1+z)=z-\frac{z^2}{2}+\frac{z^3}{3}-\cdots$$ 
	in which $\den(a_n)=n$ for every $n\in\N$. In the recent paper by Bell, Nguyen, and Zannier \cite[Equation~(7)]{BNZ22_DF2}, the authors construct a subset $S$ of $\N$ such that $\vert S\cap [1,n]\vert=O(n\log\log n/\log n)$ and for every $\beta>1$, we have:
	$$\lcm\{i:\ i\leq n, i\notin S\}<\beta^n$$
	for all sufficiently large $n$. 
	Therefore, in Theorem~\ref{thm:main 1} we cannot replace the upper bound $o(n/\log n)$ by $o(g(n))$ where $g(n)$ is any function that dominates $n\log\log n/\log n$ as $n\to\infty$.
	\end{remark}
	
	\begin{remark}\label{rem:3 after main 1}
	 While Theorem~\ref{thm:PC} as well as a couple more results in earlier work by other authors discussed later in this section are special cases of Theorem~\ref{thm:main 1} when $S=\emptyset$, the appearance of $S$ in the statement of Theorem~\ref{thm:main 1}
	 is not merely for the sake of a generalization without any further use. On the contrary, the set $S$ brings much extra flexibility and is truly necessary in certain applications. The readers who are interested in applications of Theorem~\ref{thm:main 1} are referred to Section~\ref{sec:app} in which we provide a complete characterization of when the Artin-Mazur zeta function of a linear endomorphism on a positive characteristic torus is algebraic and establish a result on the natural boundary of this zeta function in the transcendence case.  
	\end{remark}
		
	In the paper \cite{BMW14_TA}, Bell, Miles, and Ward consider the class of endomorphisms on compact abelian groups for which the number of periodic points of a given period is always finite and conjecture that the Artin-Mazur zeta function satisfies the P\'olya-Carlson dichotomy: either it is rational or it admits the circle of the radius of convergence as a natural boundary. Some partial results are given in \cite[Theorem~15]{BMW14_TA} using \cite[Lemma~17]{BMW14_TA}. More partial results for endomorphisms on abelian varieties in positive characteristic are obtained by Byszewski-Cornelissen \cite[Theorem~5.5]{BC18_DO} using a result by Royals-Ward \cite[Theorem~A1]{BC18_DO}. The power series considered in \cite[Lemma~17]{BMW14_TA} and \cite[Theorem~A1]{BC18_DO} are very special case of power series of the form 
	$f(z)=\displaystyle\sum a_nz^n$ that converge in the open unit disk and 
	each
	$a_n$ is either an integer or has the form $a_n=p_1^{c_{1,n}}\cdots p_m^{c_{m,n}}$
	where $\{p_1,\ldots,p_m\}$ is a given set of prime numbers, the $c_{i,n}$'s are rational numbers of bounded denominator, and $\vert c_{i,n}\vert=O(\log n)$. Let $L$ be the $\lcm$ of the denominators of all the $c_{i,n}$'s. Then we can apply Theorem~\ref{thm:main 1} immediately with $K=\Q\left(p_1^{1/L},\ldots,p_m^{1/L}\right)$, $S=\emptyset$, and
	\begin{equation}\label{eq:n^O(1)}
	\lcm\{\den(a_k):\ k\leq n\}=n^{O(1)}
	\end{equation}
	to conclude that either $f(z)$ is a rational function or it admits the unit circle as a natural boundary.
	
	In order to illustrate the method of this paper, we continue with the above example and explain how to establish the desired P\'olya-Carlson dichotomy. For $n\in\N$, consider the Hankel determinant
\begin{equation*}
\Delta_n=\det\begin{pmatrix} a_0 & a_{1} & \ldots & a_{n} \\
 a_{1} & a_{2} &\ldots  & a_{n+1}  \\
 \ldots  \\
 a_{n} & a_{n+1} &\ldots  & a_{2n} \end{pmatrix}.
\end{equation*}
Let $\Norm_{K/\Q}$ denote the norm function on the field $K=\Q\left(p_1^{1/L},\ldots,p_m^{1/L}\right)$. Suppose that $f(z)$ can be extended analytically beyond the open unit disk, then P\'olya's inequality \cite[Section~2.4]{BNZ22_DF2} implies that
there exists $r\in (0,1)$ such that
$$\vert \Norm_{K/\Q}(\Delta_n)\vert < r^{n^2}$$
for all sufficiently large $n$. However $\Norm_{K/\Q}(\Delta_n)$ is a rational number whose denominator is
$n^{O(n)}$ thanks to \eqref{eq:n^O(1)}. This implies that $\Delta_n=0$ for all large $n$ and Kronecker's criterion \cite[Section~2]{BNZ22_DF2} yields the rationality of $f$. This gives a much shorter proof to 
\cite[Theorem~A1]{BC18_DO} by Royals-Ward. We emphasize that the power series in our application in Section~\ref{sec:app} do not have the above form for $f(z)$ (i.e. the form
$a_n=p_1^{c_{1,n}}\cdots p_m^{c_{m,n}}$ with $\vert c_{i,n}\vert=O(\log n)$, etc.) and therefore one cannot use the earlier results in \cite{BMW14_TA} or \cite[Theorem~A1]{BC18_DO}.
		
	We are back to a general power series $f(z)=\displaystyle\sum_{n=0}^{\infty} a_nz^n\in K[[z]]$, a number field $K$, and an arbitrary subset $S$ of $\N$. In order to ``avoid'' $\den(a_k)$ for $k\in S$, our first idea taken from the paper \cite{BNZ22_DF2} is to consider
	$\displaystyle \sum_{k=0}^{\infty} P_n(k)a_kz^k$ where, roughly speaking, $P_n(z)$ is an integer-valued polynomial that vanishes on $[1,Cn]\cap S$ where $C$ is a large but fixed integer. In \cite[Theorem~3.9]{BNZ22_DF2}, this method is successful even with the much weaker condition $\vert S\cap [1,n]\vert=o(n)$ compared to the condition 
$\vert S\cap [1,n]\vert=o(n/\log n)$ imposed in Theorem~\ref{thm:main 1}. The reason is that the power series $f(z)$ in \cite{BNZ22_DF2} satisfies the powerful D-finiteness property and hence the authors of \cite{BNZ22_DF2} can afford to work with a weaker property on $S$. On the other hand, in the current paper the power series $f(z)$ need not satisfy any extra global property and a new idea is needed for the proof of Theorem~\ref{thm:main 1}. Our way to proceed is to compare the different power series
$\displaystyle \sum_k P_n(k)a_kz^k$ for \emph{different} polynomials $P_n(z)$	that vanish on $[1,Cn]\cap S$ and compare those series with the different  $\displaystyle\sum_k P_{n+1}(k)a_kz^k$	as well.

	We can use a slight variant of the above method to prove a result establishing the rationality of $f(z)$. In view of Remark~\ref{rem:1 after main 1}, a further condition on the $a_n$'s for $n\in S$ is needed for this purpose:
	\begin{theorem}\label{thm:main 2}
	Let $S$, $K$, and $f(z)$ be as in Theorem~\ref{thm:main 1}. We assume the further condition:
	\begin{equation}\label{eq:main 2 further condition}
	\den(a_n)=e^{o(n)}\quad \text{for $n\in S$ and $n\to\infty$.}
	\end{equation}
	Then either $f(z)$ admits the unit circle as a natural boundary or it is a rational function whose poles are located at the roots of unity.
	\end{theorem}
	
	The organization of this paper is as follows. In the next sections we present some preliminary results concerning Hankel determinants, P\'olya's inequality, and polynomial-exponential sequences of finitely many terms. A large part about the first two topics is taken from \cite{BNZ22_DF2} while the third topic is needed for certain patching arguments involving polynomial-exponential sequences. Then we present the proof of Theorem~\ref{thm:main 1} following the aforementioned method. The proof of Theorem~\ref{thm:main 2} uses essentially the same method with some changes; we will describe these changes carefully while skipping the similar details. Finally in Section~\ref{sec:app}, we introduce the Artin-Mazur zeta function associated to a matrix multiplication map on positive characteristic tori, provide a complete characterization of when this zeta function is algebraic, and apply Theorem~\ref{thm:main 1} to prove a result on the natural boundary in the transcendence case.
	
	\textbf{Acknowledgements.} Jason Bell was supported by NSERC grant RGPIN-2016-03632. Keira Gunn was supported by a Vanier Canada Graduate Scholarship. Khoa Nguyen and J.~C.~Saunders were supported by NSERC grant  RGPIN-2018-03770 and CRC tier-2 research stipend 950-231716. The authors wish to thank Professor Tom Ward for helpful discussions.
	
	\section{Preliminary results}\label{prelim}
	Parts of this section are taken from \cite[Section~2]{BNZ22_DF2}. The paper \cite{BNZ22_DF2} refines 
	several results in \cite{BNZ20_DF} and introduces the idea of using P\'olya's inequality for
	the series $\displaystyle\sum P_n(k)a_kz^k$ for certain polynomials $P_n(z)$. This idea also plays an important role in the proof of Theorem~\ref{thm:main 1} and Theorem~\ref{thm:main 2}.
	
	\subsection{Hankel determinants} 
	Throughout this subsection, let $k$ be a field and let $\ord$ denote the order function on $k((z))$: it is the discrete valuation with uniformizer $z$. For   
 $g(z)\in k((z))$ we use the notation $g=O(z^m)$ to mean that  $\ord g\ge m$.

For a power series $f(z)=\displaystyle\sum_{n=0}^{\infty}a_nz^n\in k[[z]]$ and 
 integers $\ell,m\ge 0$ we define the Hankel matrix  $H_{\ell,m}(f)$  and Hankel determinant $\Delta_{\ell,m}(f)$ by:
\begin{equation*}
H_{\ell,m}(f)=\begin{pmatrix} a_\ell & a_{\ell+1} & \ldots & a_{\ell+m} \\
 a_{\ell+1} & a_{\ell+2} &\ldots  & a_{\ell+m+1}  \\
 \ldots  \\
 a_{\ell+m} & a_{\ell+m+1} &\ldots  & a_{\ell+2m} \end{pmatrix},\qquad \Delta_{\ell,m}(f):=\det H_{\ell,m}(f).
\end{equation*}
 
We begin with the following:
\begin{lemma} \label{lem:H1} 
Let $\ell$ and $m$ be nonnegative integers. Then $\Delta_{\ell,m}(f)=0$ if and only if there exist polynomials $P(z),Q(z)\in k[z]$  with
$\deg P\le \ell+m-1$, $Q(z)\neq 0$, $\deg Q\le m$ and $P(z)-Q(z)f(z)=O(z^{\ell+2m+1})$.
\end{lemma}

\begin{proof} 
This is \cite[Lemma~2.1]{BNZ22_DF2}
\end{proof}

From now on, we let $\Delta_m(f)$ denote the Hankel determinant $\Delta_{0,m}(f)$. We have the following:
\begin{corollary}\label{cor:H1}
Let $m$ and $d$ be nonnegative integers. If $\Delta_{m+i}(f)=0$ for $0\leq i\leq d$ then
there exist $P(z),Q(z)\in k[z]$ with $\deg P\leq m-1$,
$Q(0)\neq 0$, $\deg Q\leq m$, $\gcd(P,Q)=1$, and
$f(z)-P(z)/Q(z)=O(z^{m+d+1})$.
\end{corollary}
\begin{proof}
This is \cite[Corollary~2.2]{BNZ22_DF2}.
\end{proof}

%

\subsection{P\'olya's inequality}
In this subsection, we let $G\subsetneq \C$ be a simply connected domain containing $0$ with conformal radius $\rho>1$ from the origin. This means
we have a (unique) conformal map $\varphi$ from $G$ onto
the disk $D(0,\rho)$ with $\varphi(0)=0$ and $\varphi'(0)=1$. Let $r\in (1,\rho)$ and let
$\Gamma=\varphi^{-1}(\partial D(0,r))$. A more general version of P\'olya's inequality treats functions with
isolated singularities, the following version
for analytic functions is sufficient for our purpose:
 
\begin{theorem}[P\'olya's inequality]\label{thm:Polya}
Let $f(z)$ be an analytic function on $G$ with the Taylor series $f(z)=\sum_{n=0}^{\infty} a_n z^n$ at $0$.  Then there 
exists a constant $C>0$ depending only
on the data $(G,r)$ such that 
$$\vert \Delta_n(f)\vert\le (n+1)! (C\Vert f\Vert_{\Gamma})^{n+1}r^{-n(n+1)}\quad \text{for every $n\in\N_0$,}$$ where $\Vert f\Vert_{\Gamma}:=\max\{\vert f(z)\vert:\ z\in\Gamma\}$.  
\end{theorem}
\begin{proof} The more general version in which $f$ can have isolated singularities is proved in \cite[pp.~121--124]{Bie55_AF}. 
\end{proof}

We need the following estimate when applying P\'olya's inequality for certain auxiliary functions $g$ constructed from $f$:
\begin{corollary}\label{cor:cor of Polya's inequality}
Let $f(z)$ be an analytic function on $G$ with the Taylor series $f(z)=\sum_{n=0}^{\infty} a_n z^n$ at $0$. There exist  constants $C_1>0$
depending on $(G,r,f)$ and $C_2\ge 1$ depending 
on $(G,r)$ such that 
for every $d\ge 0$ and $c_0,\ldots ,c_d\in \mathbb{C}$, we have
$$\vert\Delta_n(g)\vert \le (n+1)! C_1^{n+1} \left( \sum_{j=0}^d |c_j| j! C_2^{j}\right)^{n+1} r^{-n(n+1)}$$ for every $n\in \N_0$, where
$g(z):= \sum_{j=0}^d c_j z^j f^{(j)}(z)$.
\end{corollary}
\begin{proof}
This is \cite[Corollary~2.7]{BNZ22_DF2}. 
\end{proof}

For a polynomial $P(z)\in\C[z]$ of degree $d\geq 0$, we express 
$$P(z)=\sum_{i=0}^d \alpha_i \binom{z}{i}$$
for unique $\alpha_0,\ldots,\alpha_d\in \C$ so that
$$\sum_{n=0}^{\infty} P(n)a_nz^n=\sum_{i=0}^d \frac{\alpha_i}{i!}z^i f^{(i)}(z)$$
where $f(z)=\displaystyle\sum_{n=0}^{\infty} a_nz^n$. In order to apply Corollary~\ref{cor:cor of Polya's inequality}, we will use the following crude estimates for the $\vert\alpha_i\vert$'s:
\begin{lemma}\label{lem:estimate alpha_i}
Let $P(z)\in \C[z]$ with $d:=\deg(P)\geq 0$. Let $M=\max\{\vert P(i)\vert:\ i=0,1,\ldots,d\}$. Express
$P(z)=\displaystyle\sum_{i=0}^d\alpha_i\binom{z}{i}$, then we have $\vert \alpha_i\vert\leq M2^ii!$ for $0\leq i\leq d$.
\end{lemma}	
\begin{proof}
We may assume $M=1$ and proceed by induction on $i$. The cases $i=0,1$ are immediate since $\alpha_0=P(0)$
and $\alpha_1=P(1)-P(0)$. Suppose the inequality holds for $0\leq i\leq k$ with $1\leq k\leq d-1$. From the induction hypothesis, we have:
\begin{align*}
\alpha_{k+1}&=P(k+1)-\sum_{i=0}^k\alpha_i\binom{k+1}{i}\\
\vert\alpha_{k+1}\vert&\leq 1+\sum_{i=0}^k 2^i i!\binom{k+1}{i}\\
&=2^{k+1}(k+1)!\left(\frac{1}{2^{k+1}(k+1)!}+\sum_{i=0}^k\frac{1}{2^{k+1-i}(k+1-i)!}\right)\\
&<2^{k+1}(k+1)!\left(\frac{1}{8}+e^{1/2}-1\right)<2^{k+1}(k+1)!
\end{align*}
\end{proof}		
		
\subsection{Polynomial-exponential sequences}
\begin{definition}\label{def:pol-exp sequence}
Let $M\leq N$ be integers and let $(u_n)_{n=M}^N$ be a sequence of $N-M+1$ algebraic numbers. 
\begin{itemize}
\item [(a)] Let $r\in\N_0$. The sequence $(u_n)_{n=M}^N$ is called a polynomial-exponential sequence of rank $r$ if there exist $s\in\N_0$, distinct non-zero algebraic numbers $\alpha_1,\ldots,\alpha_s$, and non-zero polynomials $P_1(z),\ldots,P_s(z)\in\Qbar[z]$ such that
\begin{equation}\label{eq:u_n=sum P_i(n)alpha_i^n}
u_n=P_1(n)\alpha_1^n+\cdots+P_s(n)\alpha_s^n\quad \text{for $M\leq n\leq N$}
\end{equation}
and $\displaystyle s+\sum_{i=1}^s \deg(P_i)=r$.
\item [(b)] Let $r\in\N_0$. The sequence $(u_n)_{n=M}^N$ is called a proper polynomial-exponential sequence
of rank $r$ if it is a polynomial-exponential sequence of rank $r$ and $\displaystyle r\leq \frac{N-M+1}{2}$.
\end{itemize}
\end{definition}

\begin{remark}
In Definition~\ref{def:pol-exp sequence}(a), the case $r=0$ means that $s=0$ and the data $(\alpha_1,\ldots,\alpha_s,P_1,\ldots,P_s)$ is empty. Then \eqref{eq:u_n=sum P_i(n)alpha_i^n} means $u_n=0$ for $M\leq n\leq N$. In other words, a sequence is a (proper) polynomial-exponential sequence of rank $0$
if and only if every member of the sequence is $0$.
\end{remark}

\begin{example}
Consider the sequence $u_1=u_2=u_3=u_4=2022$. Let $d\geq 3$ and let $P(z)\in\Qbar[z]$ be a polynomial of degree $d$ such that $P(1)=P(2)=P(3)=P(4)=2022$. From $u_n=P(n)\cdot 1^n$ for $1\leq n\leq 4$ we have 
that the given sequence is a polynomial-exponential sequence of rank $d+1$. We can also take the constant polynomial $P(z)=2022$ and have that the given sequence is a \emph{proper} polynomial-exponential sequence of rank $1$.
\end{example}
		
\begin{lemma}\label{lem:uniqueness}
Let $M\leq N$ be integers and let $(u_n)_{n=M}^N$ be a sequence of algebraic numbers. Suppose $(u_n)_{n=M}^N$ is a polynomial-exponential sequence of rank $r$ and let $s\in\N_0$, distinct non-zero algebraic numbers 
$\alpha_1,\ldots,\alpha_s$, and non-zero polynomials $P_1,\ldots,P_s\in\Qbar[z]$ such that
$$u_n=P_1(n)\alpha_1^n+\cdots+P_s(n)\alpha_s^n\quad \text{for $M\leq n\leq N$}$$
and $s+\displaystyle \sum_{i=1}^s\deg(P_i)=r$. Suppose $(u_n)_{n=M}^N$ is a polynomial-exponential sequence of rank $\tilde{r}$ and let $\tilde{s}\in\N_0$, distinct non-zero algebraic numbers 
$\tilde{\alpha}_1,\ldots,\tilde{\alpha}_s$, and non-zero polynomials $\tilde{P}_1,\ldots,\tilde{P}_s\in\Qbar[z]$ such that
$$u_n=\tilde{P}_1(n)\tilde{\alpha}_1^n+\cdots+\tilde{P}_s(n)\tilde{\alpha}_s^n\quad \text{for $M\leq n\leq N$}$$
and $\tilde{s}+\displaystyle \sum_{i=1}^s\deg(\tilde{P}_i)=\tilde{r}$.
If $r+\tilde{r}\leq N-M+1$ then $r=\tilde{r}$, $s=\tilde{s}$, and up to rearrangement the pairs $(\alpha_i,P_i)$'s for $1\leq i\leq s$ coincide with the pairs $(\tilde{\alpha}_i,\tilde{P}_i)$'s for $1\leq i\leq \tilde{s}$.
\end{lemma}		
\begin{proof}
Put $w_n=\displaystyle\sum_{i=1}^s P_i(n)\alpha_i^n-\sum_{i=1}^{\tilde{s}}\tilde{P}_i(n)\tilde{\alpha}_i^n$ for $n\in\Z$. It is well-known (see \cite[p.~174]{Sch03_LR}) that $(w_n)_{n\in\Z}$ is a linear recurrence sequence of recurrence length at most $r+\tilde{r}$ with constant coefficients: there exist algebraic numbers $c_i$'s for
$0\leq i\leq r+\tilde{r}-1$ such that
\begin{equation}\label{eq:w_n recurrence}
w_n=\sum_{i=0}^{r+\tilde{r}-1}c_iw_{n+i}\quad \text{for every integer $n$.}
\end{equation}
Since $w_n=0$ for $M\leq n\leq N$ and $r+\tilde{r}\leq N-M+1$, we can run the recurrence \eqref{eq:w_n recurrence} forward and backward to conclude that $w_n=0$ for every integer $n$. Then 
\cite[Lemma~2.2]{Sch03_LR} gives that $s=\tilde{s}$ and the $(\alpha_i,P_i)$'s coincide with the
$(\tilde{\alpha}_i,\tilde{P}_i)$'s. It follows that $r=\tilde{r}$.
\end{proof}
		
\begin{corollary}\label{cor:uniqueness}
Let $M\leq N$ be integers and let $(u_n)_{n=M}^N$ be both a proper polynomial-exponential sequence of rank $r$ and a proper polynomial-exponential sequence of rank $\tilde{r}$. Then $r=\tilde{r}$. The number $s$ and the pairs $(\alpha_i,P_i)$'s for $1\leq i\leq s$ in Definition~\ref{def:pol-exp sequence} are unique up to rearrangement.
\end{corollary}		
\begin{proof}
We have $r+\tilde{r}\leq N-M+1$ since $r,\tilde{r}\leq (N-M+1)/2$. Then the corollary follows immediately from Lemma~\ref{lem:uniqueness}.
\end{proof}

\begin{remark}
One can use the equivalence notion of linear recurrence sequences with constant coefficients. Instead of the rank $r$ (which is well-defined in the proper case thanks to Corollary~\ref{cor:uniqueness}), one can use the length of a minimal recurrence relation. This is just a matter of taste and we find it more convenient to work with $r$ through the explicit expression \eqref{eq:u_n=sum P_i(n)alpha_i^n}.
\end{remark}

\begin{definition}\label{def:characteristic roots}
Let $M\leq N$ be integers and let $(u_n)_{n=M}^N$ be a proper polynomial-exponential sequence of rank $r$.
The integer $s$ in Definition~\ref{def:pol-exp sequence} is called the number of characteristic roots of $(u_n)_{n=M}^N$; this is well-defined thanks to Corollary~\ref{cor:uniqueness}.
\end{definition}

\begin{remark}\label{rem:r+ds}
Let $M\leq N$ be integers and let $(u_n)_{n=M}^N$ be a proper polynomial-exponential sequence of rank $r$ and number of characteristic roots $s$. Express $u_n$ for $M\leq n\leq N$ as in \eqref{eq:u_n=sum P_i(n)alpha_i^n}. Then for every non-zero polynomial $Q(z)\in\Qbar[z]$ of degree $d$,
the sequence $(Q(n)u_n)_{n=M}^N$ is a polynomial-exponential sequence of rank $r+ds$ since
$$Q(n)u_n=Q(n)P_1(n)\alpha_1^n+\cdots+Q(n)P_s(n)\alpha_s^n\quad \text{for $M\leq n\leq N$.}$$
Moreover, if $r+ds\leq \displaystyle\frac{N-M+1}{2}$ then $(Q(n)u_n)_{n=M}^N$ is a proper polynomial-exponential sequence of rank $r+ds$.
\end{remark}

We conclude this subsection with the following:
\begin{lemma}\label{lem:coeffs of rational functions}
Let $M\leq N$ be positive integers. Let $P(z),Q(z)\in\Qbar[z]$ with $Q(0)\neq 0$, $\gcd(P,Q)=1$, $\deg(P)\leq M-1$, and $\deg(Q)\leq M$. Let $\displaystyle\sum_{n=0}^{\infty}u_nz^n$ be the power series of $P(z)/Q(z)$ centered at $0$. We have that the sequence $(u_n)_{n=M}^N$ is a polynomial-exponential sequence of rank
$r=\deg(Q)$.
\end{lemma}
\begin{proof}
The case when $\deg(Q)=0$ is obvious since $u_n=0$ for $n\geq M$. Consider the case $\deg(Q)>0$ and assume without loss of generality that $Q(0)=1$. Write $Q(z)=(1-\alpha_1z)^{c_1}\cdots (1-\alpha_sz)^{c_s}$, express
$$\frac{P(z)}{Q(z)}=A(z)+\frac{R(z)}{Q(z)}$$
with $A,R\in\C[z]$, $\deg(A)\leq M-1$, and $\deg(R)<\deg(Q)$, and use partial fraction decomposition for
$R(z)/Q(z)$ in order to have:
$$u_n=P_1(n)\alpha_1^n+\cdots+P_s(n)\alpha_s^n\quad \text{for $n\geq M$}$$
with polynomials $P_i(z)\in\Qbar[z]$'s such that $\deg(P_i)=c_i-1$ for $1\leq i\leq s$. The resulting rank is:
$$s+\sum_{i=1}^s\deg(P_i)=\sum_{i=1}^sc_i=\deg(Q).$$
\end{proof}

\section{Proof of Theorem~\ref{thm:main 1} and Theorem~\ref{thm:main 2}}
We need the following:
\begin{lemma}\label{lem:lcm of den of g(k)}
Let $S$ be a subset of $\N$ such that $S\cap [1,n]=o(n/\log n)$ as $n\to\infty$. Let $g(z)\in\Qbar(z)$ that is not a polynomial and $n$ is not a pole of $g(z)$ for every $n\in\N\setminus S$. Then there exists $\beta>1$ such that
$$\lcm\{\den(g(k)):\ k\in \N\setminus S, k\leq n\}>\beta^n$$
for all sufficiently large $n$.
\end{lemma}
\begin{proof}
By replacing $g(z)$ by $g(z+M)$ and replacing $S$ by 
$\{s\in\N:\ s+M\in S\}$ for a large integer $M$, we may assume that 
$n$ is not a pole of $g(z)$ for every $n\in\N$. 
Let $K$ be a number field with ring of integers $\cO_K$ such that $g(z)=A(z)/B(z)$
with $A(z),B(z)\in \cO_K[z]$ having no common factor in $K[z]$; by the earlier assumption we have $B(n)\neq 0$ for every $n\in\N$. There exist $A_1(z),B_1(z)\in \cO_K[z]$ and $D\in \cO_K\setminus\{0\}$ such that
$$A_1(z)A(z)+B_1(z)B(z)=D.$$
Therefore, if $p$ is a prime number greater than $N:=\vert \Norm_{K/\Q}(D)\vert$ 
and $k\in\N$ such that $B(k)$ is divisible by a prime lying above
$p$ then $p\mid \den(g(k))$.

Let  $\tilde{B}(z)\in\Z[z]$ be the norm of $B(z)$ over $K$. For $n\in\N$, put
$$\cP_n=\{\text{prime $p$ in $[N+1,n]$}:\ p\mid \tilde{B}(k)\ \text{for some integer $k\in[1,n]$}\}\ \text{and}$$
$$\cP_n'=\{p\in\cP_n:\ p\nmid \tilde{B}(k)\ \text{for every integer $k\in [1,n]\setminus S$}\}.$$
It follows that $\displaystyle\prod_{p\in\cP_n\setminus\cP_n'}p=\left(\prod_{p\in\cP_n}p\right)/\left(\prod_{p\in\cP_n'}p\right)$ divides $\lcm\{\den(g(k)):\ k\in \N\setminus S,k\leq n\}$.

The set $\cP_n$ is exactly the set of primes in $[N+1,n]$ such that $\tilde{B}(z)$ has an integer root modulo $p$. It follows from the Chebotarev density theorem (for natural density) that there exists $C\in(0,1]$ such that $\vert \cP_n\vert > Cn/\log n$ when $n$ is large. Pick any $\beta_1\in (1,e^C)$ then we have
$$\prod_{p\in\cP_n}p > \beta_1^n$$
for all sufficiently large $n$ thanks to the prime number theorem. From the definition of $\cP_n$ and $\cP_n'$, every $p\in\cP_n'$ must divide $\tilde{B}(k)$ for some $k\in [1,n]\cap S$. Since $\vert[1,n]\cap S\vert=o(n/\log n)$, we have:
$$0<\prod_{k\in [1,n]\cap S}\vert \tilde{B}(k)\vert=n^{o(n/\log n)}=e^{o(n)}$$
and therefore
$$\prod_{p\in\cP_n'}p=e^{o(n)}$$
for all sufficiently large $n$.

Combining all the above, we may take any $\beta\in (1,\beta_1)$ and have that
$$\lcm\{\den(g(k)):\ k\in\N\setminus S, k\leq n\}>\beta^n$$
for all sufficiently large $n$. This finishes the proof.
\end{proof}

\subsection{Proof of Theorem~\ref{thm:main 1}}
Throughout this subsection, let $S$, $K$, and $f$ be as in the statement of Theorem~\ref{thm:main 1}. We assume that $f$ can be extended to an analytic function
on a simply connected domain $G$ that strictly contains the open unit disk. Let 
$\rho>1$ be the conformal radius from the origin of $G$ and fix $r\in (1,\rho)$. Let $C_1$ and $C_2$ be the constants in the conclusion of Corollary~\ref{cor:cor of Polya's inequality}.  
Fix $C_3>1$ such that:
\begin{equation}\label{eq:C3 choice}
C_3^{3[K:\Q]}<r.
\end{equation}
Fix $C_4>0$ such that
\begin{equation}\label{eq:sigma(a_n) bound}
\vert\sigma(a_n)\vert<C_4C_3^n
\end{equation}
for every embedding $\sigma$ of $K$ into $\C$; this is possible since $\sigma(f)$ converges in the open unit disk. Fix a positive number $\epsilon$ such that
\begin{equation}\label{eq:epsilon choice}
\epsilon<(\log C_3)/4.
\end{equation}

For $n\in\N$, let $d_n=\vert S\cap [1,20n]\vert$ and let $A_n(z)$ be the monic polynomial of degree $d_n$ with only simple roots that are exactly  the elements of
$S\cap [1,20n]$. We assume that $n$ is sufficiently large so that 
\begin{equation}\label{eq:d_n < epsilon n/logn}
d_n\leq \epsilon \frac{n}{\log n}.
\end{equation}
For $k\in\N_0$ with $k\leq \epsilon n/\log n$, put
\begin{equation}\label{eq:Pnk=z^kAnk}
P_{n,k}(z)=z^kA_n(z).
\end{equation}
Then we have:
\begin{equation}\label{eq:Pnk(z) bound}
\max\{\vert P_{n,k}(z)\vert:\ z\in [0,20n]\}\leq (20n)^{k+d_n}\leq (20n)^{2\epsilon n/\log n}<C_3^n
\end{equation}
when $n$ is sufficiently large thanks to \eqref{eq:epsilon choice}. By Stirling's formula, we have:
\begin{equation}\label{eq:Stirling}
\log ((d_n+k)!)\leq \log(\lfloor 2\epsilon n/\log n\rfloor!)<3\epsilon n
\end{equation}
for sufficiently large $n$ and for $k\in\N_0$ with $k\leq \epsilon n/\log n$. By Lemma~\ref{lem:estimate alpha_i}, we can express
\begin{equation}\label{eq:Pnk and alphanki}
P_{n,k}(z)=\sum_{i=0}^{k+d_n} \alpha_{n,k,i}\binom{z}{i}
\end{equation}
with
\begin{equation}\label{eq:alphanki bound}
\vert \alpha_{n,k,i}\vert \leq C_3^n 2^{d_n+k}(d_n+k)!\leq C_3^n 2^{2\epsilon n/\log n} e^{3\epsilon n}<C_3^{2n}
\end{equation}
thanks to \eqref{eq:epsilon choice} and \eqref{eq:Stirling}.

Now we consider the power series
$$\sum_{\ell=0}^{\infty} P_{n,k}(\ell)a_{\ell}z^{\ell}=\sum_{i=0}^{d_n+k} \frac{\alpha_{n,k,i}}{i!}f^{(i)}(z),$$
its Hankel matrix
\begin{align*}
H_{n,k,m}:&=H_{0,m}\left(\sum P_{n,k}(\ell)a_{\ell}z^{\ell}\right)\\
&=\begin{pmatrix} P_{n,k}(0)a_0 & P_{n,k}(1)a_{1} & \ldots & P_{n,k}(m)a_{m} \\
P_{n,k}(1)a_{1} & P_{n,k}(2)a_{2} &\ldots  & P_{n,k}(m+1)a_{m+1}  \\
 \ldots  \\
 P_{n,k}(m)a_{m} & P_{n,k}(m+1)a_{m+1} &\ldots  & P_{n,k}(2m)a_{2m} \end{pmatrix},
\end{align*}
and Hankel determinant $\Delta_{n,k,m}=\det(H_{n,k,m})$ for integers $m\in [n,10n]$. 
		
By Corollary~\ref{cor:cor of Polya's inequality}, \eqref{eq:Pnk and alphanki}, and \eqref{eq:alphanki bound}, we have:		
\begin{align}\label{eq:Deltankm bound}
\begin{split}
\vert \Delta_{n,k,m}\vert&\leq (m+1)!C_1^{m+1}\left((d_n+k+1)C_3^{2n}C_2^{d_n+k}\right)^{m+1}r^{-m(m+1)}\\
&< (C_3^3r^{-1})^{m(m+1)}
\end{split}
\end{align}
when $n$ is sufficiently large, $k<\epsilon n/\log n$, and $m\in [n,10n]$; the last inequality in \eqref{eq:Deltankm bound} follows from the fact that $C_3^{m(m+1)}$ dominates the remaining factor
$(m+1)!C_1^{m+1}\left((d_n+k+1)C_2^{d_n+k}\right)^{m+1}$. For every embedding $\sigma$ of $K$ into $\C$, we use \eqref{eq:sigma(a_n) bound} and \eqref{eq:Pnk(z) bound} to obtain:
\begin{align}\label{eq:sigma(Deltankm) bound}
\vert \sigma(\Delta_{n,k,m})\vert \leq (n+1)!C_3^{n(m+1)}C_4^{m+1}C_3^{m(m+1)}<C_3^{3m(m+1)}.
\end{align}
 
 Combining \eqref{eq:Deltankm bound} and \eqref{eq:sigma(Deltankm) bound}, we have:
 \begin{equation}\label{eq:Norm of Deltankm}
 \vert\Norm_{K/\Q}(\Delta_{n,k,m})\vert< \left(C_3^{3[K:\Q]}r^{-1}\right)^{m(m+1)}.
 \end{equation}
 Now we let
 \begin{align}\label{eq:Lnkm}
 \begin{split}
 L_{n,k,m}:&=\lcm\{\den(P_{k,n}(\ell)a_{\ell}):\ \ell\leq 2m,\ell\notin S\}\\
 &\leq \lcm\{\den(a_{\ell}):\ \ell\leq 2m,\ell\notin S\}\\
 &<\left(r^{1/[K:\Q]}C_3^{-3}\right)^m
 \end{split}
 \end{align}	
 since $r^{1/[K:\Q]}C_3^{-3}>1$ (see~\eqref{eq:C3 choice}) and the property \eqref{eq:thm main 1}. Since $P_{k,n}(\ell)a_{\ell}=0$ for $\ell\leq 2m$ and $\ell\in S$, we have that 
$L_{n,k,m}^{[K:\Q](m+1)}\vert \Norm_{K,Q}(\Delta_{n,k,m})\vert$ is a natural number that is less than $1$ thanks to \eqref{eq:Norm of Deltankm} and \eqref{eq:Lnkm}. We have proved the following:
\begin{proposition}
For every sufficiently large $n$, for $0\leq k\leq \epsilon n/\log n$, and for $m\in [n,10n]$, we have
$\Delta_{n,k,m}=0$.
\end{proposition}

By Corollary~\ref{cor:H1}, there exist $B_{n,k}(z)$ and $C_{n,k}(z)$ in $K[z]$ with $\deg(B_{n,k})\leq n-1$, $C_{n,k}(0)\neq 0$, $\deg(C_{n,k})\leq n$, and $\gcd(B_{n,k},C_{n,k})=1$ such that
\begin{equation}
\sum_{\ell=0}^{\infty} P_{n,k}(\ell)a_{\ell}z^{\ell}-\frac{B_{n,k}(z)}{C_{n,k}(z)}=O(z^{10n+1}).
\end{equation}

By Lemma~\ref{lem:coeffs of rational functions}, the sequence $(P_{n,k}(\ell)a_{\ell})_{\ell=n}^{10n}$
is a proper polynomial-exponential sequence of rank 
\begin{equation}\label{eq:rnk}
r_{n,k}= \deg(C_{n,k})\leq n.	
\end{equation}
Let $s_{n,k}$ denote the number of characteristic roots of the sequence $(P_{n,k}(\ell)a_{\ell})_{\ell=n}^{10n}$.
Recall our definition $P_{n,k}(z)=z^kA_n(z)$ from \eqref{eq:Pnk=z^kAnk}. Our next step is to compare the sequences
$(P_{n,0}(\ell)a_{\ell})_{\ell=n}^{10n}=(A_n(\ell)a_{\ell})_{\ell=n}^{10n}$
as $n$ varies. For this step, first we need to compare the $(P_{n,k}(\ell)a_{\ell})_{\ell=n}^{10n}$ when $n$ is fixed and $k$ varies. While stronger results can be obtained in a similar manner, the following inequality is enough for
our purpose:

\begin{proposition}\label{prop:rn0+n'sn0}
When $n$ is sufficiently large, we have
\begin{equation}
r_{n,0}+\lfloor \epsilon n/\log n\rfloor s_{n,0}\leq n.
\end{equation}
\end{proposition}
\begin{proof}
Put $n'=\lfloor \epsilon n/\log n\rfloor$. We prove by induction on integers $k\in [0,n']$ that
$$r_{n,0}+ks_{n,0}\leq n.$$
The case $k=0$ is simply \eqref{eq:rnk}. Suppose the desired inequality holds for $k\leq n'-1$. Then Remark~\ref{rem:r+ds} gives that 
$(P_{n,k+1}(\ell)a_{\ell})_{\ell=n}^{10n}=(\ell^{k+1}A_n(\ell)a_{\ell})_{\ell=n}^{10n}$ is a polynomial-exponential sequence of rank 
$$r_{n,0}+(k+1)s_{n,0}\leq n+s_{n,0}\leq n+r_{n,0}\leq 2n$$
where the first inequality follows from the induction hypothesis.
On the other hand, $(P_{n,k+1}(\ell)a_{\ell})_{\ell=n}^{10n}$ is a (proper) polynomial-exponential sequence of rank $r_{n,k+1}\leq n$. Since there are $9n+1\geq 3n$ many terms in our sequence, Lemma~\ref{lem:uniqueness} implies:
$$r_{n,0}+(k+1)s_{n,0}=r_{n,k+1}\leq n.$$
We finish the proof by the principle of induction.
\end{proof}

\begin{notation}
We write $r(n):=r_{n,0}$ and $s(n):=s_{n,0}$ to avoid towers of subscripts in the following notation. For a sufficiently large integer $n$, let $\beta_{n,1},\ldots,\beta_{n,s(n)}$ be distinct non-zero algebraic numbers and let $Q_{n,1},\ldots,Q_{n,s(n)}$ be non-zero polynomials in $\Qbar[z]$ such that
\begin{equation}\label{eq:betani and Qni}
A_n(\ell)a_{\ell}=Q_{n,1}(\ell)\beta_{n,1}^{\ell}+\cdots+Q_{n,s(n)}(\ell)\beta_{n,s(n)}^{\ell}\quad \text{for $n\leq \ell\leq 10n$}.
\end{equation}
The pairs $(\beta_{n,i},Q_{n,i})$'s for $1\leq i\leq s(n)$ are unique up to rearrangement thanks to Corollary~\ref{cor:uniqueness}.
\end{notation}

\begin{notation}\label{not:Dn and En}
Write in simplest form $\displaystyle \frac{A_{n+1}(z)}{A_n(z)}=\frac{D_n(z)}{E_n(z)}$ with $\gcd(D_n,E_n)=1$. From our definition of the $A_n(z)$'s, we have $E_n(z)=1$ and $D_n(z)$ is the monic polynomial with only simple roots and these roots are exactly the elements of $S\cap [20n+1,20n+20]$.
\end{notation}

\begin{proposition}\label{prop:s(n)=s(n+1)}
Let $n$ be a sufficiently large integer, we have
$s(n)=s(n+1)$ and the pairs
$(\beta_{n+1,i},Q_{n+1,i}(z))$'s for $1\leq i\leq s(n+1)$ coincide with the pairs
$(\beta_{n,i},D_n(z)Q_{n,i}(z))$'s for $1\leq i\leq s(n)$ up to rearrangement.
\end{proposition}
\begin{proof}
From $A_{n+1}=D_nA_n$ and \eqref{eq:betani and Qni}, we have:
$$A_{n+1}(\ell)a_{\ell}=D_n(\ell)Q_{n,1}(\ell)\beta_{n,1}^{\ell}+\cdots+D_n(\ell)Q_{n,s(n)}(\ell)\beta_{n,s(n)}^{\ell}\quad \text{for $n\leq \ell\leq 10n$}$$
making the sequence $(A_{n+1}(\ell)a_{\ell})_{\ell=n+1}^{10n}$ a polynomial-exponential sequence of rank 
$$r(n)+\deg(D_n)s(n)\leq r(n)+20s(n)\leq n$$
thanks to Proposition~\ref{prop:rn0+n'sn0}. This same sequence is also a polynomial-exponential sequence of rank $r(n+1)\leq n+1$ by using the instance of \eqref{eq:betani and Qni} for $n+1$. We get the desired result thanks to Lemma~\ref{lem:uniqueness}.
\end{proof}

\begin{notation}
We let $s$ denote the common value of $s(n)$ for all sufficiently large $n$. We also rearrange the $(\beta_{n,i},Q_{n,i})$'s for $1\leq i\leq s$ so that we may assume
$(\beta_{n,i},D_nQ_{n,i})=(\beta_{n+1,i},Q_{n+1,i})$ for $1\leq i\leq s$ for all sufficiently large $n$.
\end{notation}

\begin{remark}
Suppose there is a sufficiently large $n_0$ such that $s(n_0)=0$, then $s(n)=0$ for \emph{all} sufficiently large $n$ thanks to Proposition~\ref{prop:s(n)=s(n+1)}. This means $A_n(\ell)a_{\ell}=0$ for 
$n\leq \ell\leq 10n$ for all sufficiently large $n$. Since the roots of $A_n$ are in $S$, we conclude that $a_{\ell}=0$ for all sufficiently large $\ell\notin S$. Theorem~\ref{thm:main 1} follows. In the following, we assume that $s>0$.
\end{remark}

\begin{corollary}
There exist distinct non-zero algebraic numbers $\beta_1,\ldots,\beta_s$
and non-zero rational functions $R_1(z),\ldots,R_s(z)\in\Qbar(z)$
such that every $n\in\N\setminus S$ is not a pole of $R_i(z)$ for $1\leq i\leq s$ and
\begin{equation}
a_{\ell}=R_1(\ell)\beta_1^{\ell}+\cdots+R_s(\ell)\beta_s^{\ell}
\end{equation}
for every sufficiently large $\ell$ that is not in $S$.
\end{corollary}
\begin{proof}
For $1\leq i\leq s$, let $\beta_i$ be the common value of the $\beta_{n,i}$ for all large $n$. From $\displaystyle\frac{Q_{n+1,i}}{Q_{n,i}}=D_n=\frac{A_{n+1}}{A_n}$ for $1\leq i\leq s$ and for all large $n$, we have $\displaystyle \frac{Q_{n+1,i}}{A_{n+1}}=\frac{Q_{n,i}}{A_n}$. We now let $R_i(z)$ be the common value of this latter quotient for all large $n$.
\end{proof}

\begin{proposition}\label{prop:betai's are roots of unity}
The $\beta_1,\ldots,\beta_s$ are roots of unity.
\end{proposition}
\begin{proof}
Suppose that at least one of the $\beta_i$'s is not a root of unity. Note that it is possible to have that $\beta_i/\beta_j$ is a root of unity for some $1\leq i\neq j\leq s$. After restricting to an arithmetic progression 
$$\cA=\{n_0+n\theta:\ n\in\N\}$$ 
if necessary, we have non-zero algebraic numbers $\gamma_1,\ldots,\gamma_t$ and non-zero rational functions 
$U_1(z),\ldots,U_t(z)\in\Qbar(z)$ with the following properties:
\begin{itemize}
\item $a_{n_0+\ell\theta}=U_1(\ell)\gamma_1^{\ell}+\cdots U_t(\ell)\gamma_t^{\ell}$
for all sufficiently large $\ell$ such that $n_0+\ell\theta\notin S$.

\item $\gamma_i/\gamma_j$ is not a root of unity for $1\leq i\neq j\leq t$.

\item At least one of the $\gamma_i$'s is not a root of unity.
\end{itemize}

Enlarge $K$ so that all the $\gamma_i$'s are in $K$ and all the $U_i(z)$'s are in $K(z)$. Since one of the $\gamma_i$'s is not a root of unity, there must be a place $v$ of $K$ such that
$$M:=\max\{\vert \gamma_1\vert_v,\ldots,\vert\gamma_t\vert_v\}>1.$$
Pick any $M_1\in (1,M)$. By \cite[Section~2]{KMN19_AA}, we have
$$\vert a_{n_0+\ell\theta}\vert_v=\left\vert U_1(\ell)\gamma_1^{\ell}+\cdots U_t(\ell)\gamma_t^{\ell}\right\vert_v>M_1^{\ell}$$
for all sufficiently large $\ell$ such that $n_0+\ell\theta\notin S$.

If $v$ is an archimedean place, the condition that $\sigma(f)$ converges in the open unit disk for every embedding $\sigma$ of $K$ into $\C$ is violated.  If 
$v$ is a non-archimedean place, the condition
$\lcm\{\den(a_k): k\leq n,k\notin S\}=e^{o(n)}$ is violated. We arrive at a contradiction either way and this finishes the proof.
\end{proof}

\begin{proof}[Completion of the proof of Theorem~\ref{thm:main 1}]
Since the $\beta_1,\ldots,\beta_s$ are roots of unity, there exist
$\theta\in\N$ and rational functions $V_0(z),\ldots,V_{\theta-1}(z)\Qbar(z)$ that have no poles in $\N\setminus S$ such that
$$a_{i+\ell\theta}=V_i(i+\ell\theta)$$
for $0\leq i\leq \theta-1$ and for all sufficiently large $\ell$ such that $i+\ell\theta\notin S$. Since the $a_n$'s are in $K$, we have that the $V_i$'s are in $K(z)$. By Lemma~\ref{lem:lcm of den of g(k)} and the given condition \eqref{eq:thm main 1}, we have that $V_i$ is a polynomial for $0\leq i\leq \theta-1$. 
Then $\displaystyle\sum_{i=0}^{\theta-1}\sum_{\ell=0}^{\infty} V_i(i+\ell\theta)z^{i+\ell\theta}$ is a rational function whose poles are located at the roots of unity
and the coefficient of $z^n$ is equal to $a_n$ for all sufficiently large $n\in\N\setminus S$. By changing the coefficients of $z^n$ for the first finitely many $n$, we obtain the desired $\displaystyle\sum b_nz^n\in K[[z]]$ 
as in the conclusion of Theorem~\ref{thm:main 1}.
\end{proof}		
	
\subsection{Proof of Theorem~\ref{thm:main 2}}
The proof of Theorem~\ref{thm:main 2} follows the same steps in the above proof of Theorem~\ref{thm:main 1}. Therefore we explain the changes and skip the similar details. As before, we assume that $f(z)$ can be extended analytically to a simply connected domain that strictly contains the unit disk.

For $n\in\N$, if $S\cap [n,20n]$ is non-empty we define $S_n$ to be the singleton consisting of the smallest element in $S\cap [n,20n]$, otherwise we define $S_n=\emptyset$. Let $\tilde{A}_n$ be the monic polynomial with only simple roots and these roots are exactly the elements of $(S\cap [1,20n])\setminus S_n$. The degree $\tilde{d}_n:=\deg(\tilde{A}_n)$ is either $d_n$ or $d_n-1$ depending on whether $S_n$ is empty or a singleton. We have
\begin{equation}\label{eq:n is not a root}
\tilde{A}_n(n)\neq 0\quad \text{for every $n\in\N$}
\end{equation}
for the simple reason: if $n\notin S$ then obviously $n\notin (S\cap [1,20n])\setminus S_n$ and if  $n\in S$ then $S_n=\{n\}$ and hence $n\notin (S\cap [1,20n])\setminus S_n$.

For $0\leq k\leq \lfloor\epsilon n/\log n\rfloor$, let $\displaystyle\tilde{P}_{n,k}=z^k\tilde{A}_n(z)$. Then we consider the Hankel matrix
\begin{align*}
\tilde{H}_{n,k,m}=\begin{pmatrix} \tilde{P}_{n,k}(0)a_0 & \tilde{P}_{n,k}(1)a_{1} & \ldots & \tilde{P}_{n,k}(m)a_{m} \\
\tilde{P}_{n,k}(1)a_{1} & \tilde{P}_{n,k}(2)a_{2} &\ldots  & \tilde{P}_{n,k}(m+1)a_{m+1}  \\
 \ldots  \\
 \tilde{P}_{n,k}(m)a_{m} & \tilde{P}_{n,k}(m+1)a_{m+1} &\ldots  & \tilde{P}_{n,k}(2m)a_{2m} \end{pmatrix}
\end{align*}
and Hankel determinant $\tilde{\Delta}_{n,k,m}=\det(\tilde{H}_{n,k,m})$ for integers $m\in [n,10n]$. Similarly, let
\begin{equation*}
\tilde{L}_{n,k,m}:=\lcm\{\den(\tilde{P}_{k,n}(\ell)a_{\ell}):\ \ell\leq 2m,\ell\notin S\}.
\end{equation*}

If $S_n=\emptyset$, by the same reasoning as before we have that $\displaystyle\tilde{L}_{n,m,k}^{[K:\Q](m+1)}\vert \Norm_{K/\Q}(\tilde{\Delta}_{n,k,m})\vert$ is a natural number that is less than $1$ and hence $\Delta_{n,k,m}=0$ when $n$ is sufficiently large.
If $S_n=\{N\}$ is a singleton and this means $N$ is the smallest number in $S\cap [n,20n]$ then we have that
$$\displaystyle(\tilde{L}_{n,m,k}\den(a_N))^{[K:\Q](m+1)}\vert \Norm_{K/\Q}(\tilde{\Delta}_{n,k,m})\vert$$
is a natural number that is less than $1$ and hence $\Delta_{n,k,m}=0$ by using the additional assumption
$\den(a_N)=e^{o(N)}=e^{o(n)}$ in \eqref{eq:main 2 further condition} and similar estimates as before.

In either case, we have $\Delta_{n,k,m}=0$ when $n$ is sufficiently large, $0\leq k\leq \lfloor\epsilon n/\log n\rfloor$, and $m\in [n,10n]$. Hence
$$(\tilde{P}_{n,0}(\ell)a_{\ell})_{\ell=n}^{10n}=(\tilde{A}_n(\ell)a_{\ell})_{\ell=n}^{10n}$$
is a proper polynomial-exponential sequence of rank $\tilde{r}(n)\leq n$
and number of characteristic roots $\tilde{s}(n)$. We have
$$\tilde{r}(n)+\lfloor\epsilon n/\log n\rfloor\tilde{s}_n\leq n$$
by similar arguments used in the proof of Proposition~\ref{prop:rn0+n'sn0}. Express
\begin{equation}\label{eq:tilde expression}
\tilde{A}_n(\ell)a_{\ell}=\tilde{Q}_{n,1}(\ell)\tilde{\beta}_{n,1}^{\ell}+\cdots+\tilde{Q}_{n,\tilde{s}(n)}(\ell)\tilde{\beta}_{n,\tilde{s}(n)}^{\ell}\quad \text{for $n\leq \ell\leq 10n$}
\end{equation}
as in \eqref{eq:betani and Qni}.

Write in simplest form $\displaystyle\frac{\tilde{A}_{n+1}(z)}{\tilde{A}_n(z)}=\frac{\tilde{D}_n(z)}{\tilde{E}_n(z)}$. Previously in Notation~\ref{not:Dn and En}, we have that $E_n(z)=1$ and $D_n(z)$ has degree at most $20$. In the current case, 
we have $\deg(E_n)\leq 1$ and $\deg(D_n)\leq 21$. This is to take into the possibility that both $n,n+1\in S$ and then $n$ (respectively $n+1$) is a root of $\tilde{A}_{n+1}$
(respectively $\tilde{A}_n$) and is not a root of $\tilde{A}_{n}$ (respectively $\tilde{A}_{n+1})$); in this case we have $\tilde{E}_n(z)=z-n$ while $\tilde{D}_n(z)$ has only simple roots that are $n$ and elements of $S\cap [20n+1,20n+20]$.

By similar arguments to Proposition~\ref{prop:s(n)=s(n+1)}, we have
$\tilde{s}(n)=\tilde{s}(n+1)$ and the pairs
$(\tilde{\beta}_{n+1,i},\tilde{E}_n(z)\tilde{Q}_{n+1,i}(z))$'s for $1\leq i\leq \tilde{s}(n+1)$ coincide with the pairs
$(\tilde{\beta}_{n,i},\tilde{D}_n(z)\tilde{Q}_{n,i}(z))$'s for $1\leq i\leq \tilde{s}(n)$ up to rearrangement. Then we let $\tilde{s}$ denote the common value of the $\tilde{s}(n)$ for large $n$ and make the arrangement so that
$(\tilde{\beta}_{n+1,i},\tilde{E}_n(z)\tilde{Q}_{n+1,i}(z))=(\tilde{\beta}_{n,i},\tilde{D}_n(z)\tilde{Q}_{n,i}(z))$ for $1\leq i\leq \tilde{s}$ for all large $n$.
Let $\tilde{\beta}_i$ be the common value of the $\tilde{\beta}_{n,i}$ for $1\leq i\leq \tilde{s}$. From $\tilde{E}_n\tilde{Q}_{n+1,i}=\tilde{D}_n\tilde{Q}_{n,i}$, we have 
$$\displaystyle\frac{\tilde{Q}_{n+1,i}}{\tilde{Q}_{n,i}}=\frac{\tilde{D}_n}{\tilde{E}_n}=\frac{\tilde{A}_{n+1}}{\tilde{A}_n}$$
and therefore 
$$\frac{\tilde{Q}_{n+1,i}}{\tilde{A}_{n+1}}=\frac{\tilde{Q}_{n,i}}{\tilde{A}_n}$$
for $1\leq i\leq \tilde{s}$ and large $n$. For $1\leq i\leq \tilde{s}$, let $\tilde{R}_i(z)\in \Qbar(z)$ be the common rational function $\tilde{Q}_{n,i}/\tilde{A}_n$
for large $n$. From \eqref{eq:n is not a root} and \eqref{eq:tilde expression}, we have:
$$a_n=\tilde{R}_1(n)\tilde{\beta}_1^n+\cdots+\tilde{R}_{\tilde{s}}(n)\tilde{\beta}_{\tilde{s}}^n\quad \text{for every sufficiently large $n$.}$$

We prove that the $\tilde{\beta}_i$'s are roots of unity as in Proposition~\ref{prop:betai's are roots of unity}. Then we partition $\N_0$ into
congruence classes modulo a $\theta\in\N$ so that $a_n$ is given by the value at $n$ of rational function on each congruence class. Then we apply Lemma~\ref{lem:lcm of den of g(k)} to conclude that those rational functions must be polynomials and finish the proof.
		
	\section{An application to the Artin-Mazur zeta functions for linear endomorphisms on positive characteristic tori}\label{sec:app}
	Throughout this section, let $F$ be the finite field of order $q$ and characteristic $p$. Let
	$\Z_{F}=F[t]$ be the polynomial ring over $F$, $\Q_{F}=F(t)$, and 
	$$\R_{F}=F((1/t))=\left\{\sum_{i\leq m} a_it^i:\ m\in\Z,\ a_i\in F\ \text{for}\ i\leq m\right\}.$$
	The field $\R_{F}$ is equipped with the discrete valuation 
	$$v: \R_{F}\rightarrow \Z\cup\{\infty\}$$ 
	given by $v(0)=\infty$ and 
	$v(x)=-m$ where $x=\displaystyle\sum_{i\leq m} a_it^i$ with $a_m\neq 0$; in fact $\R_{F}$ is the completion of $\Q_{F}$ with respect to this valuation. Let $\vert\cdot\vert_F$ denote the non-archimedean absolute value $\vert x\vert_F=q^{-v(x)}$ for $x\in\R_{F}$. We fix an algebraic closure of $\R_F$ and the   absolute value $\vert\cdot\vert_F$ can be extended uniquely 
	to the algebraic closure \cite[pp.~131--132]{Neu99_AN}. 
	
	Let $\T_{F}=\R_{F}/\Z_{F}$ and let $\pi:\ \R_{F}\rightarrow \T_{F}$ be the quotient map. Every element $\alpha\in \T_{F}$ has the unique preimage $\tilde{\alpha}\in \R_{F}$ of the form
	$$\tilde{\alpha}=\sum_{i\leq -1} a_it^i.$$ 
	This yields a homeomorphism 
	$\T_{\bF}\cong \displaystyle\prod_{i\leq -1}F$ of compact abelian groups. 	
	The analytic number theory, more specifically the theory of characters and $L$-functions, on $\T_F$ has been studied since at least 1965 in work of Hayes \cite{Hay65_TD}. For a recent work in the ergodic theory side, we refer the readers to the paper by Bergelson-Leibman \cite{BL16_AW} and its reference in which the authors establish a Weyl-type equidistribution theorem.
	

	
	Let $f:\ X\rightarrow X$ be  a map from a topological space $X$ to itself. For each
	$k\geq 1$, let $N_k(f)$ denote the number of \emph{isolated} fixed points of $f^k$. Assume that $N_k(f)$ is 
	finite for every $k$, then one can define the Artin-Mazur zeta function \cite{AM65_OP}:
	$$\zeta_f(z)=\exp\left(\sum_{k=1}^{\infty} \frac{N_k(f)}{k} z^k\right).$$
	When $X$ is a compact differentiable manifold and $f$ is a smooth map such that $N_k(f)$ grows at most exponentially in $k$, the  question of whether $\zeta_f(z)$ is algebraic is stated in \cite{AM65_OP}. The rationality of $\zeta_f(z)$ when $f$ is an Axiom A diffeomorphism is established by Manning \cite{Man71_AA} after earlier work by Guckenheimer \cite{Guc70_AA}. 
	On the other hand, when $X$ is an algebraic variety defined over a finite field and $f$ is the Frobenius morphism,  the function $\zeta_f(z)$ is precisely the classical zeta function of the variety $X$ and its rationality is conjectured by Weil \cite{Wei49_NO} and first established by 
	Dwork \cite{Dwo60_OT}. For the dynamics of a univariate rational function, rationality of $\zeta_f(x)$ is established by  Hinkkanen in characteristic zero \cite{Hin94_ZF} while Bridy \cite{Bri12_TO,Bri16_TA} obtains both rationality and \emph{transcendence} results over positive characteristic when $f$ belongs to certain special families of rational functions. More recently, Byszewski and Cornelissen \cite[Theorem~4.3]{BC18_DO} settles the algebraicity problem for the Artin-Mazur zeta function associated to endomorphisms on abelian varieties over positive characteristic; the stronger problem concerning the natural boundary of this zeta function in the transcendence case is settled \cite[Theorem~5.5]{BC18_DO} under the assumption of the unique dominant root of a related linear recurrence sequence.
	
	Let $d$ be a positive integer and let $A\in M_d(\Z_F)$ be a $d\times d$-matrix with entries in $\Z_F$. We use the same notation $A$ to denote the multiplication-by-$A$ map from $\T_F^d$ to itself. We will show that $N_k(A)<\infty$ for every $n$ and hence one can define the Artin-Mazur zeta function $\zeta_A(z)$. In this section, we resolve the algebraicity problem for $\zeta_A(z)$: we provide a complete characterization and an explicit formula when $\zeta_A(z)$ is algebraic. Moreover, we apply Theorem~\ref{thm:main 1} to establish a natural boundary result in the transcendence case as predicted in the general conjecture of Bell-Miles-Ward \cite{BMW14_TA}. We need a couple of definitions before stating our result.

	Let $E$ be a finite extension of $\R_F$. Let 
	$$\cO_E:=\{\alpha\in E:\ \vert\alpha\vert_F\leq 1\},$$ 
	$$\cO_E^{*}=\{\alpha\in E:\ \vert\alpha\vert_F=1\},\ \text{and}$$
	$$\fp_E:=\{\alpha\in K:\ \vert\alpha\vert_F <1\}$$
	respectively denote the valuation ring, unit group, and maximal ideal. In particular: 
	$$\cO:=\cO_{\R_F}=F[[1/t]]\ \text{and}\ 
	\fp:=\fp_{\R_F}=\displaystyle\frac{1}{t} F[[1/t]]=\left\{\sum_{i\leq -1}a_i t^i:\ a_i\in F\ \forall i\right\}.$$ 
	Note that $\fp$ is the compact open subset of $\R_F$ that is both the open ball of radius $1$ and closed ball of radius $1/q$ centered at $0$. The field $\cO_E/\fp_E$ is a finite extension of 
	$\cO/\fp=F$ and the degree of this extension is called the inertia degree of $E/\R_F$ 
	\cite[p.~150]{Neu99_AN}. Let $\delta$ be this inertia degree, then $\cO_E/\fp_E$ is isomorphic to the finite field $\GF(q^{\delta})$. By applying Hensel's lemma \cite[pp.~129--131]{Neu99_AN} for the polynomial $X^{q^{\delta}-1}-1$, we have that $E$ contains all the roots of $X^{q^{\delta}-1}-1$. These roots together with $0$ form a unique copy of $\GF(q^{\delta})$ in $E$ called the Teichm\"uller representatives. This allows us to regard $\GF(q^{\delta})$ as a subfield of $E$; in fact $\GF(q^{\delta})$ is exactly the set of all the roots of unity in $E$ together with $0$. For every $\alpha\in \cO_E$, we can express uniquely:
	\begin{equation}\label{eq:alpha0 and alpha1}
	\alpha=\alpha_{(0)}+\alpha_{(1)}
	\end{equation}
	where $\alpha_{(0)}\in \GF(q^{\delta})$ and $\alpha_{(1)}\in \fp_E$.
	 
	\begin{definition}
	Let $\alpha$ be algebraic over $\R_F$ such that $\vert\alpha\vert_F\leq 1$. Let $E$ be a finite extension of $\R_F$ containing $\alpha$. 
	We call $\alpha_{(0)}$ and $\alpha_{(1)}$ in \eqref{eq:alpha0 and alpha1} respectively the constant term and $\fp$-term of $\alpha$; they are independent of the choice of $E$. When $\vert\alpha\vert_F=1$, the order of $\alpha$ modulo $\fp$ means the order of $\alpha_{(0)}$ in the multiplicative group
	$GF(q^{\delta})^*$ where $\delta$ is the inertia degree of $E/\R_F$; this is independent of the choice of $E$ as well. In fact, this order is the smallest positive integer $n$ such that $\vert \alpha^n-1\vert_F <1$.	
	\end{definition}
	
	We have the following:
	\begin{theorem}\label{thm:zeta}
	Let $A\in M_d(\Z_F)$ and put 
	$\displaystyle r(A)=\prod_{\lambda}\max\{1,\vert\lambda\vert\}$
	where $\lambda$ ranges over all the $d$ eigenvalues of $A$. Among the $d$ eigenvalues of $A$, let $\mu_1,\ldots,\mu_M$ be all the eigenvalues that are roots of unity and let $\eta_1,\ldots,\eta_N$ be all the eigenvalues that have absolute value $1$ and are not roots of unity. For $1\leq i\leq M$, let $m_i$ denote the order of $\mu_i$ modulo $\fp$. For $1\leq i\leq N$, let $n_i$ denote the order of $\eta_i$ modulo $\fp$.
	We have:
	\begin{itemize}
	\item [(a)] Suppose that for every $j\in \{1,\ldots,N\}$, there exists $i\in\{1,\ldots,M\}$ such that
	$m_i\mid n_j$. Then $\zeta_A(z)$ is algebraic and 
	$$\zeta_A(z)=(1-r(A)z)^{-1}\prod_{1\leq \ell\leq M} \prod_{1\leq i_1<i_2<\ldots<i_{\ell}\leq M}R_{A,i_1,\ldots,i_{\ell}}(z)$$
	where $\displaystyle R_{A,i_1,\ldots,i_{\ell}}(z):=\left(1-\left(r(A)z\right)^{\lcm(m_{i_1},\ldots,m_{i_\ell})}\right)^{(-1)^{\ell+1}/\lcm(m_{i_1},\ldots,m_{i_\ell})}$.

	\item [(b)] Otherwise suppose there exists $j\in\{1,\ldots,N\}$ such that for every $i\in\{1,\ldots,M\}$, we have $m_i\nmid n_j$. 
	Then both $\displaystyle\sum_{k=1}^{\infty}N_k(A)z^k$ and $\zeta_A(z)$ converge in the open disk
	$\{z\in\C:\ \vert z\vert < 1/r(A)\}$ and they admit the circle of radius $1/r(A)$ as a natural 
	boundary.  
	 Consequently, the function $\zeta_A(z)$ is transcendental.
	\end{itemize}
	\end{theorem}
	
	\begin{remark}\label{rem:1 after thm zeta}
	When $\det(A)\neq 0$, the quantity $r(A)$ in Theorem~\ref{thm:zeta} is $e^{h(A)}$ where $h(A)$ is the entropy of the endomorphism $A$. This can be proved using a straightforward adaptation of arguments in the classical case of $\R^d/\Z^d$ \cite{Wal82_AI,VO16_FO,GNS21_EO}. 
	\end{remark}
	
	\begin{remark}\label{rem:2 after thm zeta}
	We allow the possibility that any (or even both) of $M$ and $N$ to be $0$. When $N=0$, the condition in (a) is vacuously true and $\zeta_A(z)$ is algebraic in this case. When $N=0$ and $M=0$ meaning that none of the eigenvalues of $A$ has absolute value $1$, the product
	$\displaystyle\prod_{1\leq j\leq M}$ in (a) is the empty product and $\zeta_A(z)=\displaystyle\frac{1}{1-r(A)z}$. When $M=0$ and $N>0$, the condition in (b) is vacuously true and the conclusion in (b) holds.
	\end{remark}
	
	Our results are quite different from results in work of Baake-Lau-Paskunas \cite{BLP10_AN}. In \cite{BLP10_AN}, the authors prove that the zeta functions of endomorphisms of the classical torus 
	$\R^d/\Z^d$ are always rational. In our setting, we have cases when the zeta function is rational, transcendental, or algebraic irrational:
	\begin{example}\label{eg:algebraic irrational}
	Let $F=\GF(7)$ and let $A$ be the diagonal matrix with diagonal entries $\alpha,\beta\in \GF(7)^*$ where $\alpha$ has order $2$ and $\beta$ has order $3$. Then
	$$\zeta_A(z)=\frac{(1-z^2)^{1/2}(1-z^3)^{1/3}}{(1-z)(1-z^6)^{1/6}}$$
	is algebraic irrational.
	\end{example}
	
	\begin{remark}
	In \cite{BMW14_TA}, the authors consider automorphisms $T$ on compact abelian groups $X$ such that $T^k$ has finitely many fixed points for every $k\in\N$ and conjecture that the Artin-Mazur zeta function satisfies the P\'olya-Carlson dichotomy. In our setting, the property that $A^k$ has only finitely many fixed points for every $k$ is equivalent to the property that none of the eigenvalues of $A$ is a root of unity. In other words, $M=0$ in Theorem~\ref{thm:zeta}. As explained in Remark~\ref{rem:2 after thm zeta},
	we have that either $\zeta_A(z)=\displaystyle\frac{1}{1-r(A)z}$ or $\zeta_A(z)$ admits the circle of radius $1/r(A)$ as a natural boundary.
	\end{remark}

	First we derive a formula for $N_k(A)$ which is well-known 
	in the classical case of $\R^d/\Z^d$ \cite{BLP10_AN}:
	\begin{lemma}\label{lem:N_k(A) formula}
	Let $B\in M_d(\Z_F)$. The number of isolated fixed points $N_1(B)$ of the multiplication-by-$B$ map
	$$B:\ \T_F^d\rightarrow \T_F^d$$
	is $\vert \det(B-I)\vert_F$. Consequently $N_k(A)=\vert\det(A^k-I)\vert_F$ for every $k\geq 1$.
	\end{lemma}
	\begin{proof}
	When $\det(B-I)=0$, there is a non-zero $x\in\R_F^d$ such that $Bx=x$. Then for any 
	fixed point $y\in \T_F^d$, the points $y+cx$ for $c\in \R_F$ are fixed. By choosing $c$ to be in an arbitrarily small neighborhood of $0$, we have that $y$ is not isolated. Hence $N_1(B)=0$.
	
	Suppose $\det(B-I)\neq 0$. There is a 1-1 correspondence between the set of fixed points of 
	$B$ and  the set $\Z_F^d/(B-I)\Z_F^d$. Since $\Z_F$ is a PID, we obtain the Smith Normal Form of $B-I$ that is a diagonal matrix with entries $b_1,\ldots,b_d\in\Z_F\setminus\{0\}$
	and a $\Z_F$-basis $x_1,\ldots,x_d$ of $\Z_F^d$ so that $b_1x_1,\ldots,b_dx_d$
	is a $\Z_F$-basis of $(B-I)\Z_F$. Therefore the number of fixed points of $B$ is:
	$$\prod_{i=1}^d \card (\Z_F/b_i\Z_F)=\prod_{i=1}^d\vert b_i\vert_F=\vert \det(B-I)\vert_F.$$	
	\end{proof}
	
	We fix once and for all a finite extension $E$ of $\R_F$ containing all the eigenvalues of $A$ and let $\delta$ be the inertia degree of $E/\R_F$. 
	For each $\mu_i$ in the  (possibly empty) 
	multiset $\{\mu_1,\ldots,\mu_M\}$ of eigenvalues of $A$ that are roots of unity, we 
	have the decomposition:
	$$\mu_{i}=\mu_{i,(0)}+\mu_{i,(1)}$$
	with $\mu_{i,(0)}\in\GF(q^{\delta})^*$ and $\mu_{i,(1)}\in \fp_E$ as in \eqref{eq:alpha0 and alpha1}; in fact $\mu_{i,(1)}=0$ since $\mu_i$ is a root of unity. Likewise, for each $\eta_{i}$ in 
	the (possibly empty) multiset 
	$\{\eta_1,\ldots,\eta_N\}$, we have:
	$$\eta_i=\eta_{i,(0)}+\eta_{i,(1)}$$
	with $\eta_{i,(0)}\in\GF(q^\delta)^*$ and $\eta_{i,(1)}\in \fp_E\setminus\{0\}$.
	\begin{proposition}\label{prop:main formulas}
	Let $v_p$ denote the $p$-adic valuation on $\Z$. Recall that the orders of
	$\mu_{i,(0)}$ and $\eta_{j,(0)}$ in $\GF(q^{\delta})^*$ are respectively denoted $m_i$ and $n_j$
	for $1\leq i\leq M$ and $1\leq j\leq N$; each of the $m_i$'s and $n_j$'s is coprime to $p$. Let $k$ be a positive integer, we have:
	\begin{itemize}
		\item [(i)] For $1\leq i\leq M$, $\vert\mu_i^k-1\vert_F=
\left\{
	\begin{array}{ll}
		0  & \mbox{if } k\equiv 0\bmod m_i \\
		1 & \mbox{otherwise}
	\end{array}
\right.$.

	\item [(ii)] For $1\leq j\leq N$, $\vert \eta_j^k-1\vert_F=\left\{
	\begin{array}{ll}
		\vert \eta_{j,(1)}\vert_F^{p^{v_p(k)}}  & \mbox{if } k\equiv 0\bmod n_j \\
		1 & \mbox{otherwise}
	\end{array}
\right.$
	
	\item [(iii)] $\displaystyle N_k(A)=\vert \det(A^k-I)\vert_F=r(A)^k\left(\prod_{i=1}^M a_{i,k}\prod_{j=1}^N b_{j,k}\right)^{p^{v_p(k)}}$ where
	$$a_{i,k}=\left\{
	\begin{array}{ll}
		0  & \mbox{if } k\equiv 0\bmod m_i \\
		1 & \mbox{otherwise}
	\end{array}
\right.\ \text{and}\ b_{j,k}=\left\{
	\begin{array}{ll}
		\vert \eta_{j,(1)}\vert_F  & \mbox{if } k\equiv 0\bmod n_j \\
		1 & \mbox{otherwise}
	\end{array}
\right.$$
	for $1\leq i\leq M$ and $1\leq j\leq N$.
	\end{itemize}
	\end{proposition}
	\begin{proof}
	Part (i) is easy: $\mu_i^k-1=\mu_{i,(0)}^k-1$ is an element of $\GF(q^\delta)$ and it is $0$ exactly when $k\equiv 0\bmod m_i$. For part (ii), when $k\not\equiv 0\bmod n_j$, we have:
	$$\eta_j^k-1\equiv \eta_{j,(0)}^k-1\not\equiv 0\bmod \fp_K,$$
	hence $\vert \eta_j^k-1\vert_F=1$. Now suppose $k\equiv 0\bmod n_j$ but $k\not\equiv 0\bmod p$, we have:
	$$\eta_j^k-1=(\eta_{j,(0)}+\eta_{j,(1)})^k-1=k\eta_{j,(0)}^{k-1}\eta_{j,(1)}+\sum_{\ell=2}^k \binom{k}{\ell}\eta_{j,(0)}^{k-\ell}\eta_{j,(1)}^{\ell}$$
	and since $\vert k\eta_{j,(0)}^{k-1}\eta_{j,(1)}\vert_F=\vert \eta_{j,(1)}\vert_F$ is strictly larger than the absolute value of each of the remaining terms, we have:
	$$\vert \eta_j^k-1\vert_F=\vert\eta_{j,(1)}\vert_F.$$
	Finally, suppose $k\equiv 0\bmod n_j$. Since $\gcd(n_j,p)=1$, we can write $k=k_0p^{v_p(k)}$ where
	$k_0\equiv 0\bmod n_j$ and $k_0\not\equiv 0\bmod p$. We have:
	$$\vert \eta_j^k-1\vert_F=\vert \eta_j^{k_0}-1\vert_F^{p^{v_p(k)}}=\vert\eta_{j,(1)}\vert_F^{p^{v_p(k)}}$$	
	and this finishes the proof of part (ii). Part (iii) follows from parts (i), (ii), and the definition of $r(A)$.
	\end{proof}
	
	\begin{proof}[Proof of Theorem~\ref{thm:zeta}]
	First, we prove part (a). We are given that for every $j\in\{1,\ldots,N\}$, there exists $i\in\{1,\ldots,M\}$ such that $m_i\mid n_j$. 
	
	Let $k\geq 1$. If $m_i\mid k$ for some $i$ then $N_k(A)=0$ by part (c) of Proposition~\ref{prop:main formulas}. If $m_i\nmid k$ for every $i\in\{1,\ldots,M\}$ then $n_j\nmid k$ for every $j\in\{1,\ldots,N\}$ thanks to the above assumption, then we have $N_k(A)=r(A)^k$ by Proposition~\ref{prop:main formulas}. Therefore $\displaystyle\sum_{k=1}^{\infty}\frac{N_k(A)}{k}z^k$ is equal to:
	\begin{align*}
	&\sum_{\substack{k\geq 1\\m_i\nmid k\text{ for }1\leq i\leq M}} \frac{N_k(A)}{k}z^k\\
	=&\sum_{\substack{k\geq 1\\m_i\nmid k\text{ for }1\leq i\leq M}} \frac{r(A)^k}{k}z^k\\
	=&\sum_{k\geq 1}\frac{r(A)^k}{k}z^k-\sum_{\substack{k\geq 1\\m_i\mid k\text{ for some }1\leq i\leq M}} \frac{r(A)^k}{k}z^k\\
	=&-\log(1-r(A)z)\\
	&-\sum_{\ell=1}^{M}\sum_{1\leq i_1<\ldots<i_{\ell}\leq M}(-1)^{\ell-1}\sum_{\substack{k\geq 1\\ \lcm(m_{i_1},\ldots,m_{i_{\ell}})\mid k}}\frac{r(A)^k}{k}z^k\\
	=&-\log(1-r(A)z)\\
	 &+\sum_{\ell=1}^{M}\sum_{1\leq i_1<\ldots<i_{\ell}\leq M}\frac{(-1)^{\ell+1}}{\lcm(m_{i_1},\ldots,m_{i_{\ell}})}\log\left(1-(r(A)z)^{\lcm(m_{i_1},\ldots,m_{i_{\ell}})}\right)
	\end{align*}
	where the third ``$=$'' follows from the inclusion-exclusion principle. This finishes the proof of part (a).
	
	For part (b), without loss of generality, we assume that $m_i\nmid n_1$ for $1\leq i\leq M$.
	Put
	$$g(z):=\sum_{k=1}^{\infty}N_k(A)z^k.$$
	Proposition~\ref{prop:main formulas} gives that 
	$\vert N_k(A)\vert\leq r(A)^k$, hence $g$ and $\zeta_A(z)$ are convergent in the disk
	of radius $1/r(A)$. From the relation $\displaystyle g(z)=z\frac{\zeta_A'(z)}{\zeta_A(z)}$, it remains to prove that $g(z)$ admits the circle of radius $1/r(A)$ as a natural boundary.
	Assume this is not the case and we will arrive at a contradiction.  Consider
	\begin{equation}\label{eq:c_k}
		c_k:=\frac{N_k(A)}{r(A)^k}\ \text{for $k=1,2,\ldots$}
	\end{equation}
	then the series
	$$f(z):=\sum_{k=1}^\infty c_kz^k=g(z/r(A))$$
	converges in the open unit disk and does not admit the unit circle as a natural boundary. 
	
	Let $\tau$ denote the ramification index of $E/\R_F$, then each $\vert\eta_{j,(1)}\vert_F$ has the form $\displaystyle\frac{1}{q^{d_j/\tau}}$ where $d_j$ is a positive integer \cite[p.~150]{Neu99_AN}. 
	Combining this with \eqref{eq:c_k} and Proposition~\ref{prop:main formulas}(iii), we have that there exists a finite subset $\cE$ of $\N_0$ such that:
\begin{equation}\label{eq:c_k form}
	c_k=0\ \text{or}\ c_k=\left(\frac{1}{p^{d/\tau}}\right)^{p^{v_p(k)}}\ \text{for some $d\in \cE$} 
\end{equation}
for every $k\in\N$. Let $K=\Q(p^{1/\tau})$ and let $\cO_K$ be its ring of integers, we have $f(z)\in K[[z]]$. For every embedding $\sigma$ of $K$ into $\C$, since $\vert\sigma(c_k)\vert=\vert c_k\vert$ for every $k$ thanks to \eqref{eq:c_k form}, we have that $\sigma(f)$ converges in the open unit disk. Since $p$ is totally ramified in $K$, we extend $v_p$ uniquely to $K$ and use the same notation $v_p$ for this extension.

	In order to apply Theorem~\ref{thm:main 1}, we construct the subset $S$ of $\N$ as follows. Let $C_5=\displaystyle\frac{1}{\log p}$ so that
	\begin{equation}\label{eq:C5 choice}	
	p^{C_5}=e.
	\end{equation}
	For $\ell\in\N$, let 
	\begin{equation}\label{eq:S_ell definition}
	S_{\ell}=\left\{\text{integer}\ k\in (p^{\ell-1},p^{\ell}]:\ k\equiv 0\bmod p^{\ell-\lfloor C_5\log\ell\rfloor}\right\}
	\end{equation}
	so that
	\begin{equation}\label{eq:S_ell bound}
	\vert S_{\ell}\vert\leq p^{\lfloor C_5\log\ell\rfloor}\leq\ell\quad \text{thanks to \eqref{eq:C5 choice}.}
	\end{equation}
	
	Let $\displaystyle S:=\bigcup_{\ell=1}^{\infty} S_{\ell}$. We can easily prove that $\vert S\cap [1,n]\vert=o(n/\log n)$ as follows. Given a large integer $n$, let $U$ be the integer such that $p^{U-1}<n\leq p^U$. Then we have
	\begin{equation}
	\vert S\cap [1,n]\vert\leq \sum_{\ell=1}^U\vert S_{\ell}\vert=O(U^2)=O((\log n)^2)=o(n/\log n)
	\end{equation}
	using \eqref{eq:S_ell bound} and the choice of $U$.
	
	Let $C_6:=\max\{\lceil d/\tau\rceil:\ d\in \cE\}$ so that
	\begin{equation}\label{eq:c_k denominator bound}
	p^{C_6p^{v_p(k)}}c_k\in\cO_K\quad \text{for every $k\in\N$}
	\end{equation}
	thanks to \eqref{eq:c_k form}. Given $\beta>1$, fix a large integer $B$ such that
	\begin{equation}\label{eq:B choice}
	p^{C_6p^{1-B}}<\beta.
	\end{equation}
	Let $n\in\N$  be a large integer and let $U$ be the integer such that $p^{U-1}<n\leq p^U$. From our definition of $S$ and the $S_{\ell}$'s, we have that when $n$ is sufficiently large: 
	$$v_p(k)\leq U-B\quad \text{for every $k\in [1,n]\setminus S$}.$$
	Combining this with \eqref{eq:c_k denominator bound}, we have:
	\begin{equation}
	\lcm\{\den(a_k):\ k\in [1,n]\setminus S\}\leq p^{C_6p^{U-B}}< p^{C_6p^{1-B}n}<\beta^n	
	\end{equation}
	where the second inequality follows from $p^U<pn$ and the third inequality follows from \eqref{eq:B choice}.
	
	We can now apply Theorem~\ref{thm:main 1} to conclude that there exists a rational function whose poles are located at the roots of unity with Maclaurin series $\displaystyle\sum\gamma_nz^n$ such that
	$$c_n=\gamma_n\quad \text{for $n\in\N\setminus S$}.$$
	There is a finite collection of polynomials such that when $n$ is large, each $\gamma_n$ is the value at $n$ of a polynomial in the collection. Therefore:
	\begin{equation}\label{eq:bounded from below}
	\text{the $v_p(c_n)$'s for $n\in\N\setminus S$ are bounded from below.}
	\end{equation} 
	We now use the property $m_i\nmid n_1$ for $1\leq i\leq M$ to arrive at a contradiction as follows.

	Write 
	\begin{equation}\label{eq:prod etaj1 with nj mid n1}
	\prod_{1\leq j\leq N,n_j\mid n_1} \vert \eta_{j,(1)}\vert_F=\frac{1}{p^{D}}
	\end{equation}
	with $D\in\Q_{>0}$. Let $V$ be a large integer, then let $\ell$ be a larger integer so that
	\begin{equation}\label{eq:ell choice}
	V<\ell-\lfloor C_5\log\ell\rfloor.
	\end{equation}	
	By increasing $\ell$ further if necessary, we have that the interval $\displaystyle\left(\frac{p^{\ell-1}}{n_1p^V},\frac{p^{\ell}}{n_1p^V}\right]$
	contains a prime number 
	$$\tilde{p}>\max\{p,m_1,\ldots,m_M,n_1,\ldots,n_N\}.$$
	We now have that $n_1p^V\tilde{p}\in (p^{\ell-1},p^{\ell}]$ and $n_1p^V\tilde{p}\notin S_{\ell}$ thanks to
	\eqref{eq:ell choice} and the definition of $S_{\ell}$ in \eqref{eq:S_ell definition}. Then we have
	$n_1p^V\tilde{p}\notin S$ thanks to the definition of $S$.
	From \eqref{eq:c_k}, \eqref{eq:prod etaj1 with nj mid n1}, Proposition~\ref{prop:main formulas}(iii), the assumption $m_i\nmid n_1$ for every $i$, and our choice of $\tilde{p}$, we have:
	$$c_{n_1p^V\tilde{p}}=\frac{1}{p^{Dp^V}}\ \text{hence}\ v_p(c_{n_1p^V\tilde{p}})=-Dp^V.$$
	Since $V$ could be arbitrarily large and $D>0$, we obtain a contradiction to \eqref{eq:bounded from below}. This finishes the proof.
	\end{proof}

	\bibliographystyle{amsalpha}
	\bibliography{GeneralPC} 	
\end{document}